\documentclass[10pt]{amsart}
\usepackage{amsmath}
\usepackage{amsfonts}
\usepackage{amsthm} 
\usepackage{verbatim} 
\usepackage{amssymb}

\newtheorem{theorem}{Theorem}[section]
\newtheorem{lemma}[theorem]{Lemma}

\theoremstyle{definition}

\theoremstyle{remark}

\newcommand{\Div}{{\mathrm{Div}}}
\newcommand{\PDiv}{{\mathrm{PDiv}}}

\newcommand{\p}{\mbox{.}}
\newcommand{\Z}{\mathbb{Z}}

\newcommand{\R}{\mathbb{R}}

\newcommand{\st}{\mbox{ } | \mbox{ }}

\numberwithin{equation}{section}

\begin{document}

\title{Linear Systems on Edge-Weighted Graphs}

\author{Rodney James}
\address{Department of Mathematics, Colorado State University, Fort Collins, CO 80523}

\author{Rick Miranda}

\begin{abstract}
Let $R$ be any subring of the reals.
We present a generalization of linear systems on graphs where divisors are $R$-valued functions
on the set of vertices and graph edges are permitted to have nonegative weights in $R$.
Using this generalization, we provide an independent proof of a Riemann-Roch formula, which implies
the Riemann-Roch formula of Baker and Norine.
\end{abstract}

\maketitle

%\tableofcontents

%%%%%%%%%%%%%%%%%%%%%%%%%%%%%%%%%%%%%%%%%%%%%%%%%%%%%%%%%%%%%%%%%%%%%%%%%%%%%%%%%%%%%%%%%%%%%%%%%%%
%%%%%%%%%%%%%%%%%%%%%%%%%%%%%%%%%%%%%%%%%%%%%%%%%%%%%%%%%%%%%%%%%%%%%%%%%%%%%%%%%%%%%%%%%%%%%%%%%%%
%%%%%%%%%%%%%%%%%%%%%%%%%%%%%%%%%%%%%%%%%%%%%%%%%%%%%%%%%%%%%%%%%%%%%%%%%%%%%%%%%%%%%%%%%%%%%%%%%%%

\section{Introduction}
Let $R$ be any subring of the reals and
$G$ be a finite connected edge-weighted graph with vertex set $V=\{v_0, \ldots, v_n \}$ and nonnegative 
weight set $W=\{w_{ij} \st i,j=0, \ldots, n\}$ where each $w_{ij} \in R$.  Multiple edges and loops are 
not allowed, and we set $w_{ij}=0$ if $v_{i}$ and $v_{j}$ are not connected; otherwise, $w_{ij}>0$.  
Note that $w_{ii}=0$ and $w_{ij}=w_{ji}$.
We will define the \emph{degree} of a vertex $v_{j}$ to be
\[
\deg(v_{j}) = \sum_{i=0}^{n} w_{ij},
\]
and the parameter $g$ (the \emph{genus} of the graph) to be
\[ 
g= 1  + \sum_{i<j} w_{ij}- |V| = \sum_{i<j} w_{ij} - n.
\]
Note that if $R=\Z$, these definitions coincide with the usual definitions for
the vertex degree and genus of a multigraph where $w_{ij}$ is the number of edges 
connecting vertices $v_{i}$ and $v_{j}$.

A \emph{divisor} on $G$ is a function $D:V \rightarrow R$.
The \emph{degree} of a divisor $D$ is defined as
\[ 
\deg(D) = \sum_{v \in V} D(v).
\]
For any $x \in \R$, we say that $D > x$ if $D(v) > x$ for each $v \in V$,
and $D > D'$ if $D(v) > D'(v)$ for each $v \in V$.

The space of divisors on $G$, written $\Div(G)$, is a $R$-module and
the subset of divisors on $G$ with degree zero is denoted by $\Div^{0}(G)$.

The \emph{canonical divisor} $K$ is defined by $K(v)= \deg(v) -2$ for any $v \in V$.  

For any $j \in \{0, \ldots, n\}$, consider the divisor $H_{j}$, defined by
\[ 
H_{j}(v_{i}) = \left\{ \begin{array}{ll} \deg(v_{j}) & 
\mbox{if } v_{i}=v_{j} \\ -w_{ij} & \mbox{otherwise.} \end{array} \right. 
\]
The \emph{principal} divisors $\PDiv(G)$ are the $\Z$-linear combinations of the $H_{j}$ divisors.
Two divisors $D,D' \in \Div(G)$ are linearly equivalent, written $D \sim D'$, if and only if there is 
a $H \in \PDiv(G)$ such that $D-D' = H$. 

For each divisor $D \in \Div(G)$, we associate a \emph{complete linear system} $|D|$, 
which is defined as
\[ 
|D| = \{ D' \in \Div(G) \st D' \sim D, D' > -1 \}. 
\]
The \emph{dimension} of $|D|$ is defined as
\[ 
\ell(D) = \min_{E} \{ \deg(E) \st E \in \Div(G), E \ge 0, |D-E| = \emptyset \}. 
\]

We will show the following Riemann-Roch formula holds on $G$.
\begin{theorem} \label{RRoch}
For any divisor $D \in \Div(G)$
\[
\ell(D) - \ell(K-D) = \deg(D) + 1 -g.
\]
\end{theorem}
This theorem generalizes a similar statement for integral divisors on multigraphs proved by Baker 
and Norine in \cite{BN}.
We showed in \cite{JM} that Theorem~\ref{RRoch} follows from the result in \cite{BN}.
Here, the proof of Theorem~\ref{RRoch} is independent, relying on the following theorem.

Define the set of all divisors of degree $g-1$ with empty linear systems by 
\[
\mathcal{N}(G) = \{ D \in \Div(G) \st \deg(D)=g-1 \mbox{ and } |D| = \emptyset \}.
\]
\begin{theorem} \label{TwoConditions} \hspace{2cm}
\begin{enumerate}
\item The set $\mathcal{N}(G)$ is symmetric with respect to $K$; that is,\\ 
$N \in \mathcal{N}(G)$ if and only if $K-N \in \mathcal{N}(G)$.

\item For any $D \in \Div(G)$, $|D|=\emptyset$ if and only if there is a $N \in \mathcal{N}(G)$ such that $D \le N$.
\end{enumerate}
\end{theorem}

In the following section, we prove Theorem~\ref{RRoch}, assuming Theorem~\ref{TwoConditions}.  
The proof of Theorem~\ref{TwoConditions}, which relies on a normal form for divisors, 
up to linear equivalence, follows in the subsequent sections.

%%%%%%%%%%%%%%%%%%%%%%%%%%%%%%%%%%%%%%%%%%%%%%%%%%%%%%%%%%%%%%%%%%%%%%%%%%%%%%%%%%%%%%%%%%%%%%%%%%%
%%%%%%%%%%%%%%%%%%%%%%%%%%%%%%%%%%%%%%%%%%%%%%%%%%%%%%%%%%%%%%%%%%%%%%%%%%%%%%%%%%%%%%%%%%%%%%%%%%%
%%%%%%%%%%%%%%%%%%%%%%%%%%%%%%%%%%%%%%%%%%%%%%%%%%%%%%%%%%%%%%%%%%%%%%%%%%%%%%%%%%%%%%%%%%%%%%%%%%%

\section{Proof of Riemann-Roch}
For any $D \in \Div(G)$, define
\begin{eqnarray*}
 D^{+}(v) &=&  \max(D(v),0) \\
 D^{-}(v) &=&  \min(D(v),0).
\end{eqnarray*}
It follows directly from these definitions that for any $D \in \Div(G)$,
\[
D = D^{+} + D^{-}
\] 
and
\[
\deg(D^{+}) = -\deg((-D)^{-}).
\]
\begin{lemma} \label{DegPlus}
If statement (2) of Theorem~\ref{TwoConditions} is true, then for any $D \in \Div(G)$,
\[\ell(D)=\min_{N \in \mathcal{N}(G)} \deg((D-N)^{+}). \]
\end{lemma}
\begin{proof}
The definition of $\ell(D)$ is
\[
\ell(D) = \min_{E} \{ \deg(E) \st E \ge 0, |D-E| = \emptyset.
\]
Property (2) of Theorem~\ref{TwoConditions} implies that
\begin{eqnarray*}
\ell(D) &=& \min_{E,N} \{ \deg(E) \st E \ge 0, N \in \mathcal{N}(G), D-E \le N \}\\
 &=& \min_{E,N} \{ \deg(E) \st E \ge 0, N \in \mathcal{N}(G), E \ge D-N \} 
\end{eqnarray*}
or equivalently,
\[ 
\ell(D)= \min_{N \in \mathcal{N}(G)} \deg((D-N)^{+}).
\]
\end{proof}
\begin{lemma} \label{DegPlusKminusD}
If Theorem~\ref{TwoConditions} holds, then for any $D \in \Div(G)$,
\[
\ell(K-D) = g-1-\deg(D) + \min_{M \in \mathcal{N}(G)} \deg((D-M)^{+}).
\]
\end{lemma}
\begin{proof}
From property (2) of Theorem~\ref{TwoConditions},
\begin{eqnarray*}
\ell(K-D) &=& \min_{E} \{ \deg(E) \st E \ge 0, |K-D-E| = \emptyset \} \\
           &=& \min_{E,M} \{ \deg(E) \st E \ge 0, M \in \mathcal{N}(G), K-D-E \le M \}.
\end{eqnarray*}
If $K-D-E \le M$ for $M \in \mathcal{N}(G)$, then $D+E \ge K-M$.  
Property (1) of Theorem~\ref{TwoConditions} implies that 
$K-M \in \mathcal{N}(G)$ if and only if $M \in \mathcal{N}(G)$,
thus we have 
\begin{eqnarray*}
\ell(K-D) &=& \min_{E,M} \{ \deg(E) \st E \ge 0, M \in \mathcal{N}(G), D+E \ge M \} \\
           &=& \min_{E,M} \{ \deg(E) \st E \ge 0, M \in \mathcal{N}(G), E \ge M-D \} \\
           &=& \min_{M \in \mathcal{N}(G)} \deg((M-D)^{+}).
\end{eqnarray*}           
Since $\deg((M-D)^{+}) = \deg(M-D) - \deg((M-D)^{-})$, we have
\begin{eqnarray*}
\ell(K-D) &=& \min_{M \in \mathcal{N}(G)} \deg((M-D)^{+}) \\
&=& \min_{M \in \mathcal{N}(G)} \left( \deg(M-D) - \deg((M-D)^{-}) \right) \\
&=& \deg(M) - \deg(D) + \min_{M \in \mathcal{N}(G)} \left(  - \deg((M-D)^{-}) \right) \\
&=& g-1-\deg(D) + \min_{M \in \mathcal{N}(G)} \deg((D-M)^{+}).
\end{eqnarray*}      
\end{proof}
We now have the ingredients to prove Theorem~\ref{RRoch}.
\begin{proof} (Theorem ~\ref{RRoch})
Using Lemmas \ref{DegPlus} and \ref{DegPlusKminusD}, we have
\begin{eqnarray*}
\ell(D) - \ell(K-D) &=& \left( \min_{N \in \mathcal{N}(G)}\deg((D-N)^{+}) \right) \\
 & & - \left(g-1-\deg(D) + \min_{M \in \mathcal{N}(G)} \deg((D-M)^{+}) \right) \\
                      &=& \deg(D) -g +1 + \min_{N \in \mathcal{N}(G)}\deg((D-N)^{+}) \\
                      & & - \min_{M \in \mathcal{N}(G)} \deg((D-M)^{+}) \\
                      &=& \deg(D) -g +1.
\end{eqnarray*}                      
\end{proof}

%%%%%%%%%%%%%%%%%%%%%%%%%%%%%%%%%%%%%%%%%%%%%%%%%%%%%%%%%%%%%%%%%%%%%%%%%%%%%%%%%%%%%%%%%%%%%%%%%%%
%%%%%%%%%%%%%%%%%%%%%%%%%%%%%%%%%%%%%%%%%%%%%%%%%%%%%%%%%%%%%%%%%%%%%%%%%%%%%%%%%%%%%%%%%%%%%%%%%%%
%%%%%%%%%%%%%%%%%%%%%%%%%%%%%%%%%%%%%%%%%%%%%%%%%%%%%%%%%%%%%%%%%%%%%%%%%%%%%%%%%%%%%%%%%%%%%%%%%%%

\section{Reduced Divisors}
Let $V_{0}=V - \{v_{0}\}$.
We say that a divisor $D \in \Div(G)$ is \emph{reduced} if and only if
\begin{enumerate}
\item $D(v) > -1$ for each $v \in V_{0}$, and
\item for every $I \subset \{1, \ldots, n\}$, there is a $v \in V_{0}$ such that 
\[
( D - \sum_{j \in I} H_{j} ) (v) \le -1.
\]
\end{enumerate}

Define $\mathcal{P}(G) \subset \PDiv(G)$ to be the set of non-negative, non-zero 
$\Z$-linear combinations of the $H_{j}$ divisors for $j > 0$; that is, if $H \in \mathcal{P}(G)$
then there is a set of nonnegative integers $\{c_{1}, \ldots, c_{n}\}$ such that
\[H = \sum_{j=1}^{n} c_{j} H_{j}.\]
\begin{lemma} \label{subset}
Suppose a divisor $D(v) > -1$ for all $v \in V_{0}$, then
$D$ is reduced if and only if for every $H \in \mathcal{P}(G)$, there is a $v \in V_{0}$ such that 
\[
(D-H)(v) \le -1.
\]
\end{lemma}
\begin{proof}
Assume that $D(v) > -1$ for all $v \in V_{0}$.

If for every $H \in \mathcal{P}(G)$, there is a $v \in V_{0}$ such that $(D-H)(v) \le -1$,
then $D$ is clearly reduced, thus we need only show the converse is true.

Suppose that there exists a $H = \sum_{i=1}^{n} c_{i} H_{i} \in \mathcal{P}(G)$ such that
\[
(D-H)(v) > -1
\] 
for all $v \in V_{0}$.  
This means that for each $j=1, \ldots, n$
\[ 
D(v_j) > c_j \deg(v_j)- \sum_{i=1}^n c_i w_{ij} -1. 
\]
Let $\alpha = \max\{c_1, \ldots, c_n\}$ and for each $i=1, \ldots, n$ set
\[ 
b_i = \left\{ \begin{array}{ll} 1 & \mbox{if } c_i = \alpha \\ 0 & \mbox{otherwise}. \end{array} \right.
\]
We claim that for each $j \in \{1,\ldots,n\}$
\[
D(v_{j}) > b_j \deg(v_j) - \sum_{i=1}^n b_i w_{ij} -1.
\]
If $b_j=0$, we have
\[ 
b_j \deg(v_j) - \sum_{i=1}^n b_i w_{ij} -1 =  - \sum_{i=1}^n b_i w_{ij} -1 \le -1 < D(v_{j}).
\]
Define the index sets $A_{j}$ and $B_{j}$ as
\begin{eqnarray*}
A_j &=& \{ i>0  \st w_{ij} > 0 \mbox{ and } c_i < \alpha \}\\
B_j &=& \{ i>0  \st w_{ij} > 0 \mbox{ and } c_i = \alpha \}
\end{eqnarray*}
and note that
\[
\sum_{i=1}^{n} c_{i} w_{ij} = \sum_{i \in A_j} c_i w_{ij} + \alpha \sum_{i \in B_j} w_{ij}.
\]
If $b_j=1$, then $c_j = \alpha$ and 
\begin{eqnarray*}
D(v_{j}) & > & c_j \deg(v_j) - \sum_{i=1}^n c_i w_{ij} -1 \\
          &=& \alpha \deg(v_j) - \sum_{i \in A_j} c_i w_{ij} - \alpha \sum_{i \in B_j} w_{ij} -1 \\
          &=& \alpha ( w_{0j} + \sum_{i \in A_j} w_{ij} + \sum_{i \in B_j} w_{ij}) - \sum_{i \in A_j}
              c_i w_{ij} - \alpha \sum_{i \in B_j} w_{ij} -1 \\
          &=& \alpha w_{0j} + \sum_{i \in A_j} (\alpha - c_i) w_{ij} -1 \\
          &\ge& w_{0j} + \sum_{i \in A_j} w_{ij} -1 \p
\end{eqnarray*}
Also, we have
\begin{eqnarray*}
b_j \deg(v_j) - \sum_{i=1}^n b_i w_{ij} 
&=& \deg(v_j) - \sum_{i \in B_j} w_{ij} \\
&=& w_{0j} + \sum_{i \in A_j} w_{ij} + \sum_{i \in B_j} w_{ij} - \sum_{i \in B_j} w_{ij} \\
&=& w_{0j} + \sum_{i \in A_j} w_{ij}
\end{eqnarray*}
thus $D(v_{j}) > b_j \deg(v_j) - \sum_{i=1}^n b_i w_{ij} -1 $ for each $j=1,\ldots, n$,
hence $D$ is not reduced.
\end{proof}

Let $\Delta_{0}$ be the \emph{edge-weighted reduced Laplacian} of $G$, 
which can be represented by the $n \times n$ matrix
\[
\Delta_{0} = \left( \begin{array}{cccc} 
\deg(v_{1}) & -w_{12} & \cdots & -w_{1n} \\
-w_{12} & \deg(v_{2}) & \cdots & -w_{2n} \\
 \vdots & \vdots & \ddots & \vdots \\
-w_{1n} & -w_{2n} & \cdots & \deg(v_{n})
\end{array} \right).
\]
For $x=(x_{1},\ldots,x_{n})$ and $y=(y_{i},\ldots,y_{n})$, we say that
$x > y$ if and only if $x_{i} > y_{i}$ for each $i$; for any scalar $a \in \R$,
$x > a$ if and only if $x_{i} > a$ for each $i$; finally, we define 
\[
\max(x,y)=(\max\{x_{1},y_{1}\}, \ldots, \max\{x_{n},y_{n}\})
\] 
and 
$x > a$ if and only if $x_{i} > a$ for each $i$; finally, we define 
\[
\min(x,y)=(\min\{x_{1},y_{1}\}, \ldots, \min\{x_{n},y_{n}\}).
\]
We showed in~\cite{JM} that $\Delta_{0}$ is monotone; that is, 
for any $x \in \R^{n}$, if $\Delta_{0}x \ge 0$, then $x \ge 0$.
Monotonicity implies that $\Delta^{-1}_{0}$ exists and is nonnegative,
and that if $x,y \ge 0$ with $y \ge \Delta_{0} x$, then $\Delta_{0}^{-1}y \ge x$
(see \cite{BP}).

\begin{lemma} \label{zero-bound}
For any $z \in \R^{n}$ such that $z \ge 0$, there is a $c \in \Z^{n}$ such that $c \ge 0$ and 
$\Delta_{0}c \ge z$.
\end{lemma}
\begin{proof}
Fix $z \in \R^{n}$ with $z \ge 0$.  
Let $C_{0} = \{x \in \R^{n} \st \Delta_{0} x \ge 0\}$, 
$C_{z} = \{ x \in \R^{n} \st \Delta_{0} x \ge z \}$ and
$K = \{ x \in \R^{n} \st x \ge 0\}$.  Since $\Delta_{0}$ is monotone, 
$C_{z} \ne \emptyset$ and $C_{z} \subset C_{0} \subset K$.  

Let $x,y \in C_{0}$ and $\alpha, \beta \in \R$ with
$\alpha,\beta \ge 0$, then
\[
\Delta_{0}(\alpha x + \beta y ) = \alpha \Delta_{0}x + \beta \Delta_{0}y \ge 0,
\] 
and $\alpha x + \beta y \in C_{0}$.  Thus $C_{0}$ is a convex cone,
and since $\Delta_{0}$ is injective, $C_{0}$ has an interior.
Let $v=\Delta_{0}^{-1}z$.  For any $x \in C_{z}$, 
\[
\Delta_{0}(x-v) = \Delta_{0}x - z \ge 0
\]
so $x-v \in C_{0}$ and $C_{z}-v$ is also a convex cone,  thus $C_{z}$ is a convex affine cone 
with an interior. Hence $C_{z} \cap \Z^{n} \ne \emptyset$.
\end{proof}

Define the function $\phi:\Div(G) \rightarrow R^{n}$ as
\[
\phi(D) = (D(v_{1}), \ldots, D(v_{n})).
\]
We can represent any $H = \sum_{i=1}^{n} c_{i} H_{i} \in \PDiv(G)$ as
\[
\phi(H)=\Delta_{0}c
\]
where $c=(c_{1}, \ldots, c_{n}) \in \Z^{n}$; $H(v_{0})$ can be recovered by
\[
H(v_{0}) = \sum_{i=1}^{n} c_{i} H_{i}(v_{0}) = - \sum_{i=1}^{n} c_{i} w_{0i}.
\]
For any $d \in R^{n}$, define $\mathcal{A}(d) \subset \Z^{n}$ be
\[
\mathcal{A}(d)= \{ c \in \Z^{n} \st  c \ge 0, d - \Delta_{0} c > 0 \}.
\]
Note that if $d>0$, $\mathcal{A}(d) \ne \emptyset$ since the 
zero vector $(0, \ldots, 0) \in \mathcal{A}(d)$.

Let $D \in \Div(G)$ and again set $d=\phi(D+1)$.  
Using the above notation, it follows directly from Lemma~\ref{subset} 
that $D$ is reduced if and only if 
\begin{enumerate}
\item $d>0$, and
\item $\mathcal{A}(d) = \{ 0 \}$,  the zero vector in $\Z^{n}$.
\end{enumerate}

\begin{lemma} \label{bounded}
Let $D \in \Div(G)$ and set $d=\phi(D+1)$.  If $d>0$ then $\mathcal{A}(d)$ is bounded; 
that is, there exists $b \in \R^{n}$ such that $b \ge c$ for all $c \in \mathcal{A}(d)$.
\end{lemma}
\begin{proof}
Suppose $d>0$ and $c \in \mathcal{A}(d)$.  Since $d > \Delta_{0}c$,
since $\Delta_{0}$ is monotone, we have $b=\Delta_{0}^{-1} d \ge c$.
\end{proof}

\begin{lemma} \label{maximum}
Let $D \in \Div(G)$ and set $d=\phi(D+1)$.  If $d>0$ and $c,c' \in \mathcal{A}(d)$, then
$\max(c,c') \in \mathcal{A}(d)$.
\end{lemma}
\begin{proof}
Suppose $d>0$ and $c,c' \in \mathcal{A}(d)$,  then we have both $d - \Delta_{0}c >0$
and $d - \Delta_{0}c' >0$.
We can write the $j$th component of $\Delta_{0}c$ as
\begin{eqnarray*} 
(\Delta_{0}c)_{j} &=& c_{j} \deg(v_{j}) -  \sum_{i=1}^{n} c_{i} w_{ij} \\
                  &=&  \sum_{i=0}^{n} c_{j}w_{ij} -  \sum_{i=1}^{n} c_{i} w_{ij} \\
                  &=& c_{j}w_{0j} + \sum_{i=1}^{n} (c_{j} - c_{i})w_{ij}
\end{eqnarray*}
and similarly for $\Delta_{0}c'$.  For $d=(d_{1}, \ldots, d_{n})$, we then have
\begin{eqnarray*}
d_{j} &>& c_j w_{0j} + \sum_{i=1}^n (c_j -c_i) w_{ij}  \\
d_{j} &>& c'_j w_{0j}+ \sum_{i=1}^n (c'_j -c'_i) w_{ij}  
\end{eqnarray*}
for each $j$.
If $\max\{c_j,c'_j\}=c_j$, then 
\[ 
d_{j} >c_j w_{0j}+ \sum_{i=1}^n (c_j - \max\{c_i,c'_i\}) w_{ij} 
\]
and if $\max\{c_j,c'_j\}=c'_j$, 
\[ 
d_{j} > c'_j w_{0j} + \sum_{i=1}^n (c'_j - \max\{c_i,c'_i\}) w_{ij}.
\]
We can combine these two relations to get
\[
d_{j} > \max\{c_j,c'_j\} w_{0j} + \sum_{i=1}^n (\max\{c_j,c'_j\} - \max\{c_i,c'_i\}) w_{ij} 
\]
for each $j >0$, thus 
\[
d - \Delta_{0} \max(c,c') > 0
\] 
and $\max(c,c') \in \mathcal{A}(d)$.
\end{proof}
\begin{theorem} \label{reduced}
For any $D \in \Div(G)$ there is a unique $D_{0} \sim D$ such that $D_{0}$ is reduced.
\end{theorem}
\begin{proof}
Let $D \in \Div(G)$. By Lemma~\ref{zero-bound}, choose $c \in \Z^{n}$ so that  
$\Delta_{0}c \ge \phi(-(D^{-}))$, noting that $-(D^{-}) \ge 0$.
Set $d=\phi(D+1)+\Delta_{0}c$, which guarantees that $d>0$.  

By Lemmas \ref{bounded} and \ref{maximum},
$\hat{c}=\max\{c' \st c' \in \mathcal{A}(d)\}$ exists and is unique.
We claim that $\mathcal{A}(d-\Delta_{0}\hat{c})=\{0\}$.  
Let $c' \in \mathcal{A}(d-\Delta_{0}\hat{c})$, then 
\[
d -\Delta_{0}\hat{c} - \Delta_{0}c' = d - \Delta_{0}(\hat{c}+c') > 0,
\]
thus $\hat{c}+c' \in \mathcal{A}(d)$.  Since $\hat{c}$ is maximal in $\mathcal{A}(d)$
and $\hat{c},c' \ge 0$, we must have $c'=0$.
It then follows from Lemma~\ref{subset} that there is a unique reduced divisor $D_{0}$ such that
$\phi(D_{0}) = d - \Delta_{0}\hat{c}$, where $D_{0} \sim D$ since the translations $\Delta_{0}c$ 
and $\Delta_{0}\hat{c}$ correspond to some $H \in \PDiv(G)$.
\end{proof}

%%%%%%%%%%%%%%%%%%%%%%%%%%%%%%%%%%%%%%%%%%%%%%%%%%%%%%%%%%%%%%%%%%%%%%%%%%%%%%%%%%%%%%%%%%%%%%%%%%%
%%%%%%%%%%%%%%%%%%%%%%%%%%%%%%%%%%%%%%%%%%%%%%%%%%%%%%%%%%%%%%%%%%%%%%%%%%%%%%%%%%%%%%%%%%%%%%%%%%%
%%%%%%%%%%%%%%%%%%%%%%%%%%%%%%%%%%%%%%%%%%%%%%%%%%%%%%%%%%%%%%%%%%%%%%%%%%%%%%%%%%%%%%%%%%%%%%%%%%%

\section{Empty Linear Systems}

In this section, we will exploit properties of reduced divisors to determine the set of 
divisors which have empty linear systems.  We begin with the following property of reduced divisors.
\begin{lemma}\label{ReducedGenusLemma}
If $D \in \Div(G)$ is reduced, then
\[
\sum_{v \in V_{0}} D(v) \le g,
\]
with equality if and only if there exists a permutation 
$(j_{1},j_{2}, \ldots, j_{n})$ of $(1,2,\ldots,n)$
such that
\[ 
D(v_{j_{k}}) = \sum_{i=0}^{k-1} w_{j_{i}j_{k}} -1 
\]
for each $k=1, \ldots, n$, where $j_{0}=0$.
\end{lemma}
\begin{proof}
Suppose that $D$ is reduced, then for every
$I \subset \{1, \ldots, n\}$, there is a $j \in V_{0}$ such that
\begin{equation} \label{subcond}
D(v_{j}) \le \deg(v_{j}) - \sum_{i \in I} w_{ij} -1 = \sum_{i \notin I} w_{ij} - 1.
\end{equation}
Suppose that $I=I_{0}=\{1, \ldots, n\}$, and that that (\ref{subcond}) is satified for 
$j=j_{1} \in I_{0}$, then
\[ 
D(v_{j_{1}}) \le \sum_{i \notin I_{0}} w_{i j_{1}} - 1 = w_{0j_{1}} -1.
\]
Now let $I=I_{1}=I_{0} - \{j_{1}\}$, then (\ref{subcond}) is satified for
$j=j_{2} \in I_{1}$ so that
\[ 
D(v_{j_{2}}) \le \sum_{i \notin I_{1}} w_{ij_{2}} -1 = w_{0j_{2}} + w_{j_{1}j_{2}}-1.
\]
Similarly, for $I=I_{2}=I_{1}-\{j_{2}\}$, (\ref{subcond}) is satisfied for
$j=j_{3} \in I_{2}$ and
\[ 
D(v_{j_{3}}) \le \sum_{i \notin I_{2}} w_{ij_{3}} -1 
= w_{0j_{3}} + w_{j_{1}j_{3}} + w_{j_{2}j_{3}} -1.
\]
Continuing this process, let $I_{k} = I_{k-1} - \{j_{k}\}$ 
for $k=1, \ldots, n-1$, where $j=j_{k} \in I_{k-1}$ 
satisfies (\ref{subcond}) for $I_{k-1}$,  and we have in general
\begin{equation} \label{g-reduced}
D(v_{j_{k}}) \le \sum_{i=0}^{k-1} w_{j_{i} j_{k}} -1
\end{equation}
where $j_{0}=0$.  Note that the resulting $n$-tuple $(j_{1},j_{2}, \ldots, j_{n})$ 
is a permutation of $(1,2, \ldots, n)$.
If we rewrite (\ref{g-reduced}) as
\[ 
D(v_{j_{k}}) - \sum_{i=0}^{k-1} w_{j_{i} j_{k}} +1 \le 0 
\]
and sum over all $k$, we have
\[ 
\sum_{k=1}^{n} \left( D(v_{j_{k}}) - \sum_{i=0}^{k-1} w_{j_{i}j_{k}} + 1 \right) 
= \sum_{j=1}^{n} D(v_{j}) - \sum_{i<j} w_{ij} + n \le 0 
\]
or equivalently,
\[ 
\sum_{v \in V_{0}} D(v) \le \sum_{i<j} w_{ij} - n = g.
\]

For the equality condition, we assume again the $D$ is reduced and first note that if 
\[
D(v_{j_{k}}) = \sum_{i=0}^{k-1} w_{j_{i}j_{k}} -1
\] 
holds for some $(j_{1}, \ldots, j_{n})$ for each $k =1, \ldots, n$, then 
\[
\sum_{v \in V_{0}} D(v) = g
\]
follows directly since 
\[
\sum_{i<j} w_{ij} - n = g.
\]
For the other direction, if 
\[
\sum_{v \in V_{0}} D(v) = g,
\] 
since $D$ is reduced, 
(\ref{g-reduced}) holds at each $k$ for some permutation $(j_{1},\ldots,j_{n})$,
thus the only way that we can have equality is for
\[ 
D(v_{j_{k}}) = \sum_{i=0}^{k-1} w_{j_{i}j_{k}} -1 
\]
for each $k=1, \ldots, n$.
\end{proof}

An immediate application of Lemma~\ref{ReducedGenusLemma} and Theorem~\ref{reduced} 
gives a sufficient condition for a divisor to have a nonempty linear system.
\begin{lemma} \label{gminusone}
Let $D \in \Div(G)$.   If $\deg(D)>g-1$ then $|D| \ne \emptyset$.
\end{lemma}
\begin{proof}
Let $D$ be a divisor with $\deg(D)>g-1$, and let $D_{0}$ be the unique reduced divisor 
such that $D_{0} \sim D$ from Theorem~\ref{reduced}.
By Lemma~\ref{ReducedGenusLemma} 
\[ 
\sum_{v \in V_{0}} D_{0}(v) \le g.
\]
By assumption we have
\[ 
\deg(D)=\deg(D_{0})=D_{0}(v_{0}) + \sum_{v \in V_{0}} D_{0}(v) > g-1,
\] 
or equivalently
\[ 
D_{0}(v_{0}) > -\sum_{v \in V_{0}} D_{0}(v) +g -1,
\]
thus $D_{0}(v_{0})>-1$.  Since $D_{0}(v)>-1$ for each $v \in V_{0}$, $|D| \ne \emptyset$.
\end{proof}
\begin{lemma} \label{reduced-empty}
If $D_{0}$ be a reduced divisor, then $|D_{0}| \ne \emptyset$ if and only if $D_{0}(v_{0}) > -1$.
\end{lemma}
\begin{proof}
Let $D_{0} \in \Div(G)$ be reduced.
If $D_{0}(v_{0}) > -1$, then $D_{0}(v) > -1$ for all $v \in V$ and $|D_{0}| \ne \emptyset$.

Now assume that $|D_{0}| \ne \emptyset$, thus there is a $P \in \PDiv(G)$ such that $D_{0}+P > -1$.
Since $D_{0}$ is reduced, the only $P \in \PDiv(G)$ which would satisfy 
$D_{0}+P > -1$ must have $P(v) \ge 0$ for all $v \in V_{0}$.  Since $\deg(P)=0$, $P(v_{0}) \le 0$,
thus we must have $D_{0}> -1$ in order for $|D_{0}|$ to be nonempty.
\end{proof}

\begin{lemma} \label{N-set}
If $D_{0}$ is a reduced divisor with $\deg(D_{0})=g-1$ and $|D_{0}|=\emptyset$, then
\[ 
D_{0}(v_{j_{l}}) 
= \left\{ \begin{array}{ll} -1 & l=0 \\ \sum_{i=0}^{l-1} w_{j_{i}j_{l}} - 1 & l >0 \end{array} \right. 
\]
where $j_{0}=0$ and $(j_{1}, \ldots, j_{n})$ is a permutation of $(1, \ldots, n)$.
\end{lemma}
\begin{proof}
By Lemma~\ref{reduced-empty} $D_{0}(v_{0}) \le -1$,
thus by Lemma~\ref{ReducedGenusLemma} we then have 
\[
\sum_{i=1}^{n}D_{0}(v_{i})=g, 
\]
$D_{0}(v_{0})=-1$, and
\[ 
D_{0}(v_{j_{l}}) = \sum_{i=0}^{l-1} w_{j_{i}j_{l}} - 1 
\]
for some permutation $(j_{1}, \ldots, j_{n})$ of $(1, \ldots, n)$.
\end{proof}
We will denote the reduced divisors in Lemma~\ref{N-set} as
\[ 
\mathcal{N}_{0}(G) = \{ D \in \Div(G) \st |D|=\emptyset, \deg(D)=g-1, D \mbox{ is reduced } \} 
\subset \mathcal{N}(G), 
\]
noting that $|\mathcal{N}_{0}(G)| \le n!$.

A direct consequence of Lemma~\ref{N-set} then gives us the composition of $\mathcal{N}(G)$,
which is a lattice generated by $\mathcal{N}_{0}(G)$.
\begin{lemma} \label{N-set-full}
$\mathcal{N}(G) = 
\left\{ D \in \Div(G) \st D \sim D_{0} \mbox{ where } D_{0} \in \mathcal{N}_{0}(G) \right\}$.
\end{lemma}
\begin{proof}
If $D \in \mathcal{N}(G)$, then by Lemma~\ref{reduced} there is a $D_{0} \in \mathcal{N}_{0}(G)$ 
that is linearly equivalent to $D$.
\end{proof}

We can now prove Theorem~\ref{TwoConditions}.
%\begin{lemma} \label{K-Symmetry}
%$D \in \mathcal{N}(G)$ if and only if $K-D \in \mathcal{N}(G)$.  
%\end{lemma}
\begin{proof} (Theorem~\ref{TwoConditions})

\begin{enumerate}
\item Since any $D \in \mathcal{N}(G)$ can be written as $D=N_{0}+P$ for some $P \in \PDiv(G)$
and $N_{0} \in \mathcal{N}_{0}(G)$, it is sufficient to assume $D=N_{0}$.

By Lemma~\ref{N-set}, 
\[
N_{0}(v_{j_{0}})=-1
\] 
and 
\[
N_{0}(v_{j_{k}}) = \sum_{i=0}^{k-1} w_{j_{i}j_{k}} - 1
\]
for some permutation $(j_{1}, \ldots, j_{n})$ of $(1, \ldots, n)$ with $j_{0}=0$. Since 
\[
K(v_{i})= \sum_{j=0}^{n} w_{ij} - 2,
\] 
for $k>0$ we have
\[
(K-D)(v_{j_{k}}) = \sum_{i=0}^{n} w_{j_{i}j_{k}} - 2 - \sum_{i=0}^{k-1} w_{j_{i}j_{k}} + 1 
= \sum_{i=k}^{n} w_{j_{i}j_{k}} -1 
\]
and for $k=0$
\[ 
(K-D)(v_{j_{0}}) = \sum_{i=0}^{n} w_{j_{i}j_{0}} - 1 = \deg(v_{j_{0}}) -1.
\]
If we subtract $H_{0} \in \PDiv(G)$ from $K-D$, we have
\[
(K-D-H_{0})(v_{j_{0}}) = -1
\]
and for $k>0$,
\[
(K-D-H_{0})(v_{j_{k}}) = \sum_{i=k}^{n} w_{j_{i}j_{k}} -1 + w_{j_{0}j_{k}}.
\]
Let $l_{0}=j_{0}=0$ and $l_{k}=j_{n-k+1}$ for $k=1, \ldots, n$; then $(l_{1}, \ldots, l_{n})$ is
permutation of $(1, \ldots, n)$ and
\[
(K-D-H_{0})(v_{l_{k}}) = \sum_{i=0}^{k-1} w_{l_{i}l_{k}} -1,
\]
thus $K-D-H_{0} \in \mathcal{N}_{0}(G)$ and $K-D \in \mathcal{N}(G)$.

Now assume that $K-D \in \mathcal{N}_{0}(G)$.  Let $D'=K-D$, and from above we have 
$K-D' = D \in \mathcal{N}(G)$.
%\end{proof}
%
%\begin{theorem} \label{EmptyDivisorSet}
%If $|D| = \emptyset$, then $D \le N$ for some $N \in \mathcal{N}(G)$.
%\end{theorem}
%\begin{proof}
\bigskip

\item Let $D\in\Div(G)$ with $|D|=\emptyset$.    
By Lemma~\ref{reduced}, there is a unique reduced divisor $D_{0} \sim D$.  
Since $|D_{0}|=\emptyset$, Lemma~\ref{reduced-empty} implies that $D_{0}(v_{0}) \le -1$.
By the proof of Lemma~\ref{ReducedGenusLemma}, we have that (\ref{g-reduced}) holds for each $D_{0}(v)$
where $v \in V_{0}$, so for some permutation $(j_{1},\ldots,j_{n})$ of $(1, \ldots, n)$,
\[ 
D_{0}(v_{j_{k}}) \le \sum_{i=0}^{k-1} w_{j_{i}j_{k}} - 1 
\]
and thus $D_{0} \le N_{0}$ for one of the $N_{0} \in \mathcal{N}_{0}(G)$.
Let $P \in \PDiv(G)$ such that $D = D_{0} + P$, and let $N = N_{0} +P$.
Then we have $D \le N$ where $N \in \mathcal{N}(G)$.
\end{enumerate}
\end{proof}

%%%%%%%%%%%%%%%%%%%%%%%%%%%%%%%%%%%%%%%%%%%%%%%%%%%%%%%%%%%%%%%%%%%%%%%%%%%%%%%%%%%%%%%%%%%%%%%%%%%
%%%%%%%%%%%%%%%%%%%%%%%%%%%%%%%%%%%%%%%%%%%%%%%%%%%%%%%%%%%%%%%%%%%%%%%%%%%%%%%%%%%%%%%%%%%%%%%%%%%
%%%%%%%%%%%%%%%%%%%%%%%%%%%%%%%%%%%%%%%%%%%%%%%%%%%%%%%%%%%%%%%%%%%%%%%%%%%%%%%%%%%%%%%%%%%%%%%%%%%


\begin{thebibliography}{99}


\bibitem{BN}
Matthew Baker and Serguei Norine.
Riemann-Roch and Abel-Jacobi Theory on a Finite Graph.
\emph{Advances in Mathematics} 215 (2007).

\bibitem{BP}
Bernman, Abraham and Plemmons, Robert J.,
\emph{Nonnegative Matrices in the Mathematical Sciences},
Classics in Applied Mathematics,
SIAM, Philadelphia, PA, 1994.

\bibitem{Biggs99}
N.L. Biggs.  Chip-Firing and the Critical Group of a Graph. 
\emph{Journal of Algebraic Combinatorics} 9 (1999).

\bibitem{GK}
Andreas Gathmann and Michael Kerber.
A Riemann-Roch Theorem in Tropical Geometry.
\emph{Mathematische Zeitschrift} 259 (2008).

\bibitem{JM}
Rodney James and Rick Miranda.
A Riemann-Roch theorem for edge-weighted graphs.
arXiv:0908.1197 (2009).

\bibitem{Post}
A. Postnikov and B. Shapiro. 
Trees, parking functions, syzygies, and deformations of monomial ideals.
\emph{Transactions of the American Mathematical Society} 356 (8) (2004).

\end{thebibliography}
\end{document}